\documentclass[preprint,1p]{elsarticle}

\makeatletter
 \def\ps@pprintTitle{%
 	\let\@oddhead\@empty
 	\let\@evenhead\@empty
 	\def\@oddfoot{\footnotesize\itshape
 		{} \hfill\today}%
 	\let\@evenfoot\@oddfoot
 }
\makeatother
\usepackage[unicode]{hyperref}
\usepackage[T1]{fontenc}
\usepackage{latexsym}
\usepackage{indentfirst}
\usepackage{amsxtra}
\usepackage{amssymb}
\usepackage{amsthm}
\usepackage{amsmath}
\usepackage{amsfonts}
\usepackage{mathrsfs} 
\usepackage{xcolor}
\usepackage{color}
\usepackage{amsfonts}
\usepackage{mathtools}

\usepackage[capitalise]{cleveref}

\newtheorem{theor}{Theorem}[section]
\newtheorem{prop}[theor]{Proposition}
\newtheorem{cor}[theor]{Corollary}
\newtheorem{lemma}[theor]{Lemma}
\theoremstyle{definition} 
\newtheorem{defin}[theor]{Definition}
\newtheorem{rem}[theor]{Remark}

\newtheorem{ex}{Example}
\newtheorem{exs}{Examples}

\DeclareMathOperator{\Sym}{Sym}
\DeclareMathOperator{\id}{id}
\DeclareMathOperator{\Aut}{Aut}
\DeclareMathOperator{\Ret}{Ret}

\begin{document}

\begin{frontmatter}
	\title{One-generator skew braces and indecomposable set-theoretic solutions to the Yang-Baxter equation
%\tnoteref{mytitlenote}
}
	%\title{Some indecomposable involutive solutions of the Yang-Baxter equation of multipermutation level $2$ with non-abelian regular permutation group
%\tnoteref{mytitlenote}
%}
	\author{M. CASTELLI}
    \affiliation{University of Salento}
	\ead{ marco.castelli@unisalento.it - marcolmc88@gmail.com}
	
	%\author[*]{M.~CASTELLI	%\tnoteref{mytitlenote1}}
	%	\ead{marco.castelli@unisalento.it - marcolmc88@gmail.com}
	%\tnotetext[mytitlenote1]{The author is member of GNSAGA (INdAM).}
 
	%\author{S. RAMIREZ }
	%\ead{sramirez@dm.uba.ar}
	%\cortext[c1]{Corresponding author}
%	\ead{}
%	\author[unile]{}
%	\ead{}
%	\address[unile]{ }

\begin{abstract}
We study the class of one-generator solutions to the Yang-Baxter equation, extending some recent results concerning the classes of involutive and multipermutation solutions. Moreover we show the precise relationship between indecomposable solutions to the Yang-Baxter equation and finite one-generator skew braces, giving a positive answer to a question posed by Agata and Alicja Smoktunowicz. In the last part, we apply our results to the involutive case, and we present some numerical results involving solutions of small size. %Through the paper, we will use the language of q-cycle sets.
\end{abstract}

\begin{keyword}
\texttt{set-theoretic solution\sep Yang-Baxter equation\sep skew brace \sep indecomposable solution \sep q-cycle set }
\MSC[2020] 
16T25\sep 81R50 
\end{keyword}
% ----------------------
\end{frontmatter}

% ----------------------
\section*{Introduction}

The quantum Yang-Baxter equation first appeared in theoretical physics, in a paper by C.N. Yang \cite{yang1967}, and in statistical mechanics, in R.J. Baxter's work \cite{bax72}. To date, it is the subject of many studies of significant and ongoing interest, even beyond theoretical physics. In 1992, Drinfel'd \cite{drinfeld1992some} suggested the study of the set-theoretic version of this equation. Namely, a \emph{set-theoretic solution of the Yang-Baxter equation} on a non-empty set $X$ is a pair $\left(X,r\right)$, where 
$r:X\times X\to X\times X$ is a map such that the following identity
\begin{align*}
\left(r\times\id_X\right)
\left(\id_X\times r\right)
\left(r\times\id_X\right)
= 
\left(\id_X\times r\right)
\left(r\times\id_X\right)
\left(\id_X\times r\right)
\end{align*}
is satisfied.  
Writing a solution $(X,r)$ as $r\left(x,y\right) = \left(\lambda_x\left(y\right)\rho_y\left(x\right)\right)$, with
$\lambda_x, \rho_x$ maps from $X$ into itself, for every $x\in X$, we say that $(X, r)$ is \emph{non-degenerate} if $\lambda_x,\rho_x\in \Sym_X$, for every $x\in X$, and \emph{involutive} if $r^2=\id_{X\times X}$.\\
The papers of Gateva-Ivanova and Van Den Bergh  \cite{gateva1998semigroups} and Etingov, Schedler, and Soloviev \cite{etingof1998set} have led many mathematicians to study the involutive non-degenerate solutions, with a particular focus on the \emph{indecomposable} ones. The interest in these solutions stems from the fact that they enable the construction of other solutions, not necessarily indecomposable, through appropriate methods such as dynamical extensions and retraction-process (see \cite{etingof1998set,vendramin2016extensions} for more details). 
In this context, remarkable structural results on involutive solutions have been obtained (see, for example, \cite{dietzel2023indecomposable,etingof1998set,JePiZa20x,jedlivcka2021cocyclic} and related references). A successful approach has consisted in studying the solutions through algebraic structures, such as braces, introduced by Rump in \cite{rump2007braces}. A triple $(B,+,\circ)$ is said to be a \emph{brace} if $(B,+)$ is an abelian group, $(B,\circ)$ is a group and $$a\circ (b+c)=a\circ b-a+a\circ c$$
for all $a,b,c\in B$. These algebraic structures have played a crucial role in the study of indecomposable involutive solutions, since several properties of an indecomposable solution reflect to suitable properties of the so-called \emph{permutation braces}, as highlighted for example in \cite{castelli2022characterization,cedo2022indecomposable,cedo2022new,rump2020,smock}.
\\
In recent years, the study of non-involutive solutions has increased in interest. To investigate (not-necessarily indecomposable) non-involutive solutions Guarnieri and Vendramin in \cite{guarnieri2017skew} generalized braces to \emph{skew braces}, in which the abelianity of $(B,+)$ is not assumed. Even if involutive non-degenerate solutions remain a topic of great interest (see for example the recent papers \cite{dietzel2023indecomposable,kanrar2024decomposability} and related references), several results have provided for bijective non-degenerate solutions, and many mathematicians provided a non-involutive version of some celebrated notions and theorems. In that regard, in \cite{EtSoGu01} a notion of decomposability for non-involutive and non-degenerate solutions has been introduced. Moreover, the main result of \cite{EtSoGu01} represents a first significant step by providing a complete classification of indecomposable non-degenerate solutions $(X,r)$ having $X$ of prime size, for which the classification of the involutive ones given in \cite[Theorem 2.12]{etingof1998set} is a special case. Other remarkable results on indecomposable non-involutive solutions were obtained by Gateva-Ivanova in \cite{Ga21}, where she provided tools involving indecomposable solutions that are consistent with the ones provided in \cite{gateva2004combinatorial} for the involutive ones. Further results that arise as generalizations of the theory developed for the involutive ones have been presented by Mazzotta, Stefanelli and myself in \cite{castelli2022simplicity}, where indecomposable solutions have been constructed by dynamical extensions of q-cycle sets, by Jedlicka and Pilitowska in \cite{jedlicka2024diagonals}, where (in)decomposability of solutions has been investigated by the retraction-process, and by Colazzo, Jespers, Kubat and Van Antwerpen in \cite{colazzo2024simple}, where simple solutions have been completely characterized. Some results that have not a counterpart in the involutive setting were obtained investingating the so-called \emph{derived} indecomposable solutions (see for example the recent paper \cite{colazzo2022derived}).\\
The first aim of this paper is to study the family of finite \emph{one-generator} solutions, where a finite solution $(X,r)$ is said to be \emph{one-generator} if there exists an element $x\in X$ such that if $(Y,r_{|Y\times Y})$ is a subsolution containing $x$, then $Y=X$. These solutions will be studied in detail in \cref{studsol}, and we will show that they are related to several well-known families of solutions, such as indecomposable  solutions, multipermutation solutions, and irreducible involutive solutions, recently introduced in \cite{dietzel2025endocabling}.  \\
Our second goal is the study of the relationship between one-generator skew braces and one-generator solutions. Recall that a skew brace $B$ is said to be \emph{one-generator} if there exists $x\in B$ such that $B=B(x)$, where $B(x)$ is the smallest sub-brace of $B$ containing $x$. In the abelian setting of braces, the study of this family of braces is a very active field of research (see for example \cite{ballester2023structure,dixon2025structure,kurdachenko2024structure,rump2020one}). In \cite{smock}, Agata and Alicja Smoktunowicz showed that a one-generator brace $B$ always provides an involutive indecomposable solution $(X,r)$. Moreover, they proved that under additional assumptions, there is a kind of correspondence (see \cite[Theorem 6.5]{smock}). As noted in the Introduction of \cite{rump2020one}, actually there is not a perfect relationship between indecomposable involutive solutions and one-generator braces. As a main application of our theory involving one-generator solutions, we will provide a relationship between finite one-generator skew braces and indecomposable solutions. In that regard, the main theorem of \cref{secques} will provide a positive answer to \cite[Question 6.8]{smock}.\\
In the last part of the paper, we apply our results to braces and involutive solutions. In particular, we will discuss the existence of particular examples (and counterexamples). One-generator solutions having small size will be examined in detail, and we will give some numerical results, obtained by a GAP script based on the package \emph{YangBaxter} \cite{YBGAP}.

\section{Basic definitions and results}

In this section, we provide basic definitions and results useful for understanding the rest of the paper.

\subsection{Set-theoretic solutions and q-cycle sets}
We start by exploiting the existing one-to-one correspondence between non-degenerate solutions and non-degenerate q-cycle sets, algebraic structures introduced by Rump \cite{rump2019covering}. 

\begin{defin}
    A non-empty set $X$ endowed with two binary operations $\cdot$ and $:$ is said to be a \emph{q-cycle set} if the map $\sigma_x:X\rightarrow X, y\mapsto x\cdot y$ is bijective, for every $x\in X$, and the following conditions
\begin{align}
\left(x\cdot y\right)\cdot \left(x\cdot z\right) 
&= \left(y:x\right)\cdot \left(y\cdot z\right)\label{ug1}\tag{q1}\\
\left(x:y\right):\left(x:z\right) 
&= \left(y\cdot x\right):\left(y:z\right)\label{ug2}\tag{q2}\\
\left(x\cdot y\right):\left(x\cdot z\right) 
&= \left(y:x\right)\cdot \left(y: z\right)\label{ug3}\tag{q3}
\end{align}
hold, for all $x,y,z\in X$. Besides, $X$ is \emph{regular} if the map $\delta_x:X\rightarrow X, y\mapsto x:y$ is bijective, for every $x\in X$; \emph{non-degenerate} if $X$ is regular and the squaring maps, i.e. the maps $\mathfrak{q},\mathfrak{q'}:X\longrightarrow X$ given by
$\mathfrak{q}\left(x\right):=x\cdot x$ and $\mathfrak{q'}\left(x\right):=x: x$,
for every $x \in X$, are bijective.
\end{defin}
Thus, if $X$ is a non-degenerate q-cycle set, then $\left(X,r\right)$ is a non-degenerate bijective solution, where $r:X\times X\to X\times X$ is the map defined by 
\begin{align*}
r\left(x,y\right)= \left(\sigma_{x}^{-1}\left(y\right),\delta_{\sigma_{x}^{-1}\left(y\right)}\left(x\right)\right),
\end{align*}
for all $x,y\in X$. 
Vice versa, if $\left(X,r\right)$ is a non-degenerate bijective solution, if we set 
\begin{align*}
  x\cdot y:=\lambda_x^{-1}\left(y\right) 
  \qquad
  \text{and}
  \qquad
  x:y:=\rho_{\lambda_y^{-1}\left(x\right)}\left(y\right), 
\end{align*}
for all $x,y\in X$, then $X$ is a non-degenerate q-cycle set
(cf. \cite[Proposition 1]{rump2019covering}).
Evidently, if $X$ is a q-cycle set such that the operations $\cdot$ and $:$ coincide, then $X$ is a \emph{cycle set}. Cycle sets were introduced by Rump in \cite{rump2005decomposition} and have been extensively investigated (see, for instance, \cite{lebed2017homology,rump2022class,vendramin2016extensions} and related references). Of course, if $X$ is a cycle set, we have that $\sigma_x$ coincides with $\delta_x$, for all $x\in X$. Involutive non-degenerate solutions correspond to non-degenerate cycle sets, and the correspondence is nothing other than the previous one restricted to the involutive setting. A sub-q-cycle set of a q-cycle set $X$ is a subset $Y$ of $X$ such that with the binary operations induced by $\cdotp$ and $:$ is again a q-cycle set.\\
A function $f$ from an indecomposable q-cycle set $(X,\cdotp,:)$ to a q-cycle set $(Y,\cdotp',:')$ is said to be \emph{homomorphism} if $f(x\cdotp y)=f(x)\cdotp' f(y)$ and $f(x:y)=f(x):'f(y)$ for all $x,y\in X$. Moreover, a surjective homomorphism is said to be \emph{epimorphism} and a bijective homomorphism is called \emph{isomorphism}. As with other algebraic structures, epimorphisms are related to the notion of a congruence, where an equivalence relation $\sim$ of a q-cycle set $X$ is said to be a \emph{congruence} if $x\sim y $ and $x'\sim y'$ imply $x\cdotp x'\sim y\cdotp y'$ and $x: x'\sim y: y'$. If $X$ has a finite size, then the quotient $X/\sim$ can be endowed with a canonical q-cycle set structure for which the map $x\mapsto [x]_{\sim}$ is a q-cycle sets epimorphism. A well known congruence is the so-called \emph{retract relation}, which is given by
$$x\sim_{\Ret} y:\Longleftrightarrow \sigma_x=\sigma_y\quad and \quad \delta_x=\delta_y $$
for all $x,y\in X$. By the main result of \cite{JeP18}, the quotient $X/\sim_{\Ret}$, which we will indicate by $\Ret(X)$, is a non-degenerate q-cycle set even if $X$ has not finite size, provided $X$ non-degenerate. If $x$ is an element of $X$, we will indicate by $\Ret(x)$ the equivalence class of $x$ respect to $\sim_{\Ret}$. Now, if $X$ is a non-degenerate cycle set, we can define inductively $\Ret^n(X)$ as 
$$\Ret^0(X):=X \quad and \quad \Ret^n(X):=\Ret(\Ret^{n-1}(X)) $$
for all $n\in \mathbb{N}$. For our purposes, it is useful to introduce the following definition.

\begin{defin}
    Let $X$ be a finite q-cycle set. Then, we will call \emph{absolute retraction} of $X$ the q-cycle set $\Ret^m(X)$, where $m$ is the smallest natural number such that $\Ret^{m}(X)\cong \Ret^{m+1}(X)$.
\end{defin}

\noindent Note that the previous definition is consistent with the one given by Rump in \cite{rump2023primes}, where he considered (not necessarily finite) cycle sets. If the absolute retraction of a finite q-cycle set has size $1$, we will say that $X$ is a \emph{multipermutation} q-cycle set, while if the absolute retraction is isomorphic to $X$, we will say that $X$ is an \emph{irretractable} q-cycle set.

\begin{defin}
    Let $(X,\cdotp,:)$ be a non-degenerate q-cycle set. Then, $X$ is said to be \emph{indecomposable} if there is no non-trivial partition $Y\cup Z$ of $X$ such that $Y$ and $Z$ are sub-q-cycle sets of $X$. 
\end{defin}

\noindent By results contained in \cite{castelli2022simplicity,EtSoGu01}, we know that a finite q-cycle set $X$ is indecomposable if and only if the group generated by the set $\{\sigma_x|x\in X \}\cup \{\delta_x|x\in X \}$ acts transitively on $X$. A q-cycle set that is not indecomposable will be called \emph{decomposable}. As one can expect, indecomposable q-cycle sets correspond to indecomposable solutions.
%At first sight, the notion of non-degeneracy introduced by Rump seems different, but  \cite[Corollary 2]{rump2019covering} ensures that the two definitions are equivalent.\\
\begin{rem}
    Throughout the paper, we will use the language of q-cycle sets, but all the results can be easily translated into terms of solutions by the correspondence described above. In the context of involutive solutions, we will simply refer to cycle sets. 
\end{rem}

An important notion involving q-cycle sets is the dynamical pair, introduced in \cite{CaCaSt20}, that is a useful tool to construct new q-cycle sets. Specifically, given a q-cycle set $X$, a set $S$, two maps $\alpha:X\times X\times S\to\Sym_S$ and $\alpha':X\times X\times S\to S^S$, where $S^S$ is the set of all the maps from $S$ into itself, the pair $\left(\alpha,\alpha'\right)$ is called a \emph{dynamical pair} if the following equalities
\begin{align}\label{coci}
\alpha_{\left(x\cdot y\right),\left(x\cdot z\right)}\left(\alpha_{\left(x, y\right)}\left(s,t\right),\alpha_{\left(x,z\right)}\left(s,u\right)\right)&=\alpha_{\left(y:x\right),\left(y\cdot z\right)}\left(\alpha'_{\left(y,x\right)}\left(t,s\right),\alpha_{\left( y,z\right)}\left(t,u\right)\right) \notag\\
\alpha'_{\left(x: y\right),\left(x: z\right)}\left(\alpha'_{\left(x,y\right)}\left(s,t\right),\alpha'_{\left(x,z\right)}\left(s,u\right)\right)&=\alpha'_{\left(y\cdot x\right),\left(y:z\right)}\left(\alpha_{\left(y,x\right)}\left(t,s\right),\alpha'_{\left(y,z\right)}\left(t,u\right)\right)\\
\alpha'_{\left(x\cdot y\right),\left(x\cdot z\right)}\left(\alpha_{\left(x,y\right)}\left(s,t\right),\alpha_{\left(x,z\right)}\left(s,u\right)\right)&=\alpha_{\left(y: x\right),\left(y: z\right)}\left(\alpha'_{\left(y,x\right)}\left(t,s\right),\alpha'_{\left(y,z\right)}\left(t,u\right)\right)\notag
\end{align}
hold, for all $x,y,z\in X$ and $s,t,u\in S$.\\
As shown in \cite[Theorem 16]{CaCaSt20}, the triple $(X\times S,\cdot,:)$ where
\begin{align*}
(x,s)\cdot (y,t)
:=(x\cdot y,\alpha_{(x,y)}(s,t)) 
\quad and \quad
\left(x,s\right): \left(y,t\right):=\left(x: y,\alpha'_{\left(x,y\right)}\left(s,t\right)\right),
\end{align*}
for all $x,y\in X$ and $s,t\in S$, is a q-cycle set. 
If $X$ is regular and $\alpha'_{\left(x,y\right)}\left(s,-\right)\in\Sym_S$, for all $x,y\in X$ and $s\in S$, then  $\left(X\times S,\cdot,:\right)$ is regular. Moreover, the converse is true if $X$ and $S$ have finite order. The q-cycle set $X\times_{\alpha,\alpha'} S:=(X\times S,\cdot,:)$ is said to be a \emph{dynamical extension} of $X$ by $S$. The following result, which closes the subsection, is of crucial importance for our purposes. 

\begin{theor}[Theorem 3.3, \cite{castelli2022simplicity}]\label{cov}
If $X$ is an indecomposable q-cycle set and $p:X \to Y$ is an epimorphism from $X$ to a q-cycle set $Y$, then there exist a set $S$ and a dynamical
pair $(\alpha,\alpha')$ such that $X\cong Y\times_{\alpha,\alpha'} S$. 
\end{theor}

\subsection{Skew braces}

In this subsection, we introduce skew braces and we recall some basic relations with q-cycle sets.

\begin{defin}[Definition $1.1$, \cite{guarnieri2017skew}]
A triple $(B,+,\circ)$ with two binary operations is said to be a \emph{skew brace} if $(B,+)$ and $(B,\circ)$ are groups and $a\circ (b+c)=a\circ b-a+a\circ c$ for all $a,b,c\in B$. $(B,+)$ is called the \emph{additive group}, and $(B,\circ)$ is called the \emph{multiplicative group}. Moreover, the inverse of an element $a\in B$ in the additive (resp. multiplicative) group will be indicated by $-a$ (resp. $a^-$).
\end{defin}

\begin{ex}\label{primies}
\begin{itemize}
    \item[1)] Let $(B,+)$ be a group and $\circ$ be the binary operation on $B$ given by $a\circ b:= a+b$ for all $a,b\in B$. Then, the triple $(B,+,\circ)$ is a skew brace. 
    \item[2)] Let $B:=(\mathbb{Z}/p^2\mathbb{Z},+)$ and $\circ$ be the binary operation on $B$ given by $a\circ b:=a+b+p\cdotp a\cdotp b$ (where $\cdotp$ is the ring-multiplication of $\mathbb{Z}/p^2\mathbb{Z}$). Then, $(B,+,\circ)$ is a skew brace.
\end{itemize}
\end{ex}

\noindent Given a skew brace $B$ and $a\in B$, let us denote by $\lambda_a,\delta_a:B\longrightarrow B$ the maps from $B$ into itself defined by $\lambda_a(b):= - a + a\circ b,$ and $\delta_a(b):=a\circ b-a$ for all $b\in B$. Then, $\lambda_a,\delta_a\in\Aut(B,+)$, for every $a\in B$; and the map $\lambda$ (resp. $\delta$) from $B$ to $Aut(B,+)$ given by $\lambda(a):=\lambda_a$ (resp. $\delta(a):=\delta_a$) is a homomorphism (resp. anti-homomorphism) from $(B,\circ)$ to $\Aut(B,+)$. If $B$ is a skew brace, then the operations $\cdotp $ and $:$ given by $a\cdotp b:=\lambda_{a^{-}}(b)$ and $a:b:=\delta_{a^{-}}(b)$ give rise to a q-cycle set. Moreover, $B$ is a skew brace with abelian additive group if and only if the operations $\cdotp$ and $:$ coincide.

\begin{rem}
    Through the paper, we will refer to a skew brace with an abelian additive group as a brace. Note that, in this case, the q-cycle sets operations $\cdotp$ and $:$ associated with $B$ coincide.
\end{rem}

A subset $X$ of $B$ is said to be a \emph{cycle base} if it is a union of orbits with respect to the action of the group generated by the set $\{\lambda_a|a\in B \}\cup \{\delta_a|a\in B \}$ and additively (or equivalently multiplicatively) generates $B$. A cycle base is called \emph{transitive} if it consists of a single orbit. It is easy to show that a transitive cycle base is an indecomposable sub-q-cycle set of $B$.

From now on, for a finite skew brace $B$ and an element $x\in B$, we denote by $B(x)$ the smallest skew brace containing $x$. The following family of skew braces was considered for the first time in \cite{smock}.

\begin{defin}
    Let $B$ be a skew brace. Then, $B$ is said to be a \emph{one-generator} skew brace if there exists $x\in B$ such that $B=B(x)$.
\end{defin}

\section{One-generacy and irreducibility of q-cycle sets}\label{studsol}

In this section, we introduce the concept of one-generacy of a q-cycle set. Moreover, we will consider the special class of irreducible q-cycle sets, which is consistent with the one given in \cite{dietzel2025endocabling} for cycle sets. 

\bigskip

At first, note that if $X$ is a q-cycle set and $\{Y_I\}_I$ is a family of sub-q-cycle sets of $X$, then $\bigcap_I Y_I$ is a sub-q-cycle set of $X$. Now, let $S$ be a subset of $X$ and $\mathcal{Y}_S$ be the family of all the sub-q-cycle sets of $X$ containing $S$. Then, we can define the q-cycle set \emph{generated} by $S$, and we indicate it by $<S>$, as the sub-q-cycle set of $X$ given by
$$<S>:=\bigcap_{Y\in \mathcal{Y}_S} Y .$$
Of course, $<S>$ is the smallest sub-q-cycle set of $X$ containing $S$. If $S$ is a singleton $\{x\}$, we indicate $<S>$ simply by $<x>$.

\begin{defin}
    Let $X$ be a q-cycle set. Then, $X$ is said to be \emph{one-generator} if there exists $x\in X$ such that $X=<x>$. 
\end{defin}

Below we introduce a special class of q-cycle sets that includes the ones recently considered in \cite{dietzel2025endocabling}.

\begin{defin}
    A q-cycle set $X$ is said to be \emph{irreducible} if $\emptyset$ and $X$ are the only sub-q-cycle sets of $X$.
\end{defin}

The following result provides a concrete inductive description (in the finite case) of a sub-q-cycle set generated by a set. 

\begin{prop}\label{descr}
    Let $X$ be a finite q-cycle set and $S\subseteq  X$. Define inductively 
    $$\mathcal{C}_{0}:=S\quad and \quad \mathcal{C}_{n}:=\mathcal{C}_{n-1}\cup\{\sigma_a(b)|a,b\in \mathcal{C}_{n-1} \}\cup \{\delta_a(b)|a,b\in \mathcal{C}_{n-1} \}$$
    for all $n\in \mathbb{N}$. Then, $<S>=\bigcup_{n\in \mathbb{N}} \mathcal{C}_n $.
\end{prop}

\begin{proof}
    If $a,b\in \bigcup_{n\in \mathbb{N}} \mathcal{C}_n $, then $a\cdotp b, a:b\in \bigcup_{n\in \mathbb{N}} \mathcal{C}_n $ and by finiteness of $X$ we have also $\sigma_a^{-1}(b), \delta_a^{-1}(b)\in \bigcup_{n\in \mathbb{N}} \mathcal{C}_n $. Obviously, \cref{ug1}, \cref{ug2} and \cref{ug3} hold for all $a,b,c\in \bigcup_{n\in \mathbb{N}} \mathcal{C}_n$. Therefore, we showed that $\bigcup_{n\in \mathbb{N}} \mathcal{C}_n$ is a q-cycle set. Clearly, $S$ is contained in $\bigcup_{n\in \mathbb{N}} \mathcal{C}_n$. By induction on $n$, one can show that $\mathcal{C}_n\subseteq <S>$ for all $n\in \mathbb{N}$, hence $\bigcup_{n\in \mathbb{N}} \mathcal{C}_n\subseteq <S>$ and the statement follows.
\end{proof}

\noindent The q-cycle set generated by a singleton is always indecomposable, as we can see in the following proposition.

\begin{prop}\label{descr2}
    Let $X$ be a finite q-cycle set and $x\in  X$. Let $C_x:=\bigcup_n \mathcal{C}_n$ be as in \cref{descr} with $S:=\{x\}$. Then, $C_x$ is an indecomposable q-cycle set.
\end{prop}

\begin{proof}
    Clearly, \cref{ug1}, \cref{ug2} and \cref{ug3} hold for all $a,b,c\in C_x$. Since $C_x=\bigcup_n \mathcal{C}_n$, we show by induction on $n$ that every element of $\mathcal{C}_n$ is contained in the orbit $o(x)$ of $x$ with respect to the action of the group generated by the maps $\sigma_a$ and $\delta_a$. If $n=0$, the statement is trivial. Now, if $a,b\in \mathcal{C}_{n-1}$, $b$ is in the same orbit of $x$ by inductive hypothesis and clearly $\sigma_a(b)$ and $\delta_a(b)$ are also in the same orbit of $x$, therefore $\mathcal{C}_n\subseteq o(x)$, and hence the statement follows.
\end{proof}

\begin{exs}
   \begin{itemize}
       \item[1)] Let $X$ be the set $\mathbb{Z}/p\mathbb{Z}$. Define on $X$ the binary operations $\cdotp$ and $:$ by $x\cdotp y:=y+1$ and $x:y:=y+2$ for all $x,y\in X$. Then, $X$ is irreducible: indeed, if $Y$ is a nonempty sub-q-cycle set, $y\in Y$ and $t\in \{1,...,p-1\}$, we have $y+t:=\sigma_y^{t}(y)\in Y$, therefore $Y=X$.
       \item[2)] Every decomposable q-cycle set is reducible: indeed, a decomposition of $X$ induces two proper nonempty sub-q-cycle sets.
       \item[3)] There are indecomposable q-cycle sets that are not one-generator. Such examples will be provided in \cref{ultsec}.
   \end{itemize} 
\end{exs}

The orbits with respect to the squaring maps allow to provide several generators of a one-generator q-cycle set.

\begin{prop}
    Let $X$ be a finite one-generator q-cycle set and suppose that $x\in X$ satisfies $X=<x>$. If $k$ is a natural number, then $\mathfrak{q}^k(x)$ and $\mathfrak{q'}^k(x)$ are also generators of $X$. In particular, if $|X|>1$ and $\mathcal{G}$ is the set of generators of $X$, then $\mathcal{G}>1$. 
\end{prop}

\begin{proof}
At first, note that if $z\in X$ and $ y\in <z>$, then $\mathfrak{q}^m(y)$ and $\mathfrak{q'}^m(y)$ are elements of $<z>$, for all $m\in \mathbb{N}$.   Now, if $m$ is such that $\mathfrak{q}^{km}=\mathfrak{q}'^{km}=id_X$, then $x=\mathfrak{q}^{km}(x)$ and $x=\mathfrak{q'}^{km}(x)$, therefore $x\in <\mathfrak{q}^k(x)>$ and $x\in <\mathfrak{q}'^k(x)>$, and hence the first part of the statement follows. If $\mathcal{G}=1$, then by the first part we must have $x\cdotp x=x$ and $x:x=x$, hence by \cref{descr} we have  $X=\{x\}$.
\end{proof}
    
\noindent The following proposition is an easy criterion to state when a q-cycle set is irreducible.

\begin{prop}\label{critirr}
    Let $X$ be a finite q-cycle set. Then, $X$ is irreducible if and only if $<x>=X$ for all $x\in X$.
\end{prop}

\begin{proof}
    If $X$ is irreducible, then $<x>$ must be equal to $X$, for all $x\in X$. Conversely, if $<x>=X$ for all $x\in X$ and $Y$ is a non-empty sub-q-cycle set of $X$, then for every $y\in Y$ we obtain that $X=<y>\subseteq Y$, hence $X=Y$.
\end{proof}
    
\noindent The property of being one-generator is preserved by epimorphic images.

\begin{prop}\label{riduc2}
    Let $X,Y$ be finite q-cycle sets, $x\in X$, and $p:X\rightarrow Y$ an epimorphism. If $X$ is generated by $x$, then $Y$ is generated by $p(x)$. 
\end{prop}

\begin{proof}
The statement follows by a standard calculation.
\end{proof}

\noindent Now, we provide two results involving irreducible q-cycle sets that extend two recent results provided in \cite[Propositions 3.2 and 3.3]{dietzel2025endocabling} in the context of cycle sets.

\begin{cor}\label{riduc}
    Let $X,Y$ be q-cycle sets and $p:X\rightarrow Y$ an epimorphism. If $X$ is irreducible, then $Y$ is also irreducible.
\end{cor}

\begin{proof}
It is similar to the proof of \cite[Proposition 3.2]{dietzel2025endocabling}.
\end{proof}

\begin{prop}
    Let $X$ be a finite q-cycle set and suppose that the subgroup of $Sym(X)$ generated by $\{\mathfrak{q},\mathfrak{q}'\}$ acts transitively on $X$. Then, $X$ is irreducible.
\end{prop}

\begin{proof}
    If $Y\subseteq X$ and $y\in Y$, then $\mathfrak{q}(y)$ and  $\mathfrak{q}'(y)$ belong to $Y$, and this fact, together with our hypothesis, implies $X=Y$.
\end{proof}

\noindent The following results provide relations between the generators of $X$ and the ones of its retraction $\Ret(X)$.

\begin{theor}\label{linkret2}
    Let $X$ be a finite q-cycle set. If $X$ is one-generator, then so is $\Ret(X)$.\\
    Conversely, if $X$ is indecomposable and $\Ret(X)$ is generated by  an element $i\in \Ret(X)$, then $X$ is generated by any element $x\in X$  such that $\Ret(x)=i$.
\end{theor}

\begin{proof}
    The first part follows by \cref{riduc2}. Now, suppose that $X$ is indecomposable and that $\Ret(X)$ is one-generator. By \cref{cov} applied with the epimorphism $\Ret$, without loss of generality we can suppose that $X=I\times_{\alpha,\alpha'} S$, where $S$ is a set and $I$ is a q-cycle set isomorphic to $\Ret(X)$, and that $\Ret$ is the projection of $I\times S$ on the first component. Now, suppose that $I=<i>$ for some element $i$ of $I$ and let $(i,s),(j,t)\in X$. Since $X$ is indecomposable and finite, there exist $k\in \mathbb{N}$, $(x_1,s_1),...,(x_k,s_k)\in X$ and $\gamma^{1},...,\gamma^{k}\in \{\sigma,\delta \}$ such that $\gamma^{1}_{(x_1,s_1)}...\gamma^{k}_{(x_k,s_k)}(i,s)=(j,t).$ Now, since $I$ is generated by $i$, we have that each $x_j\in <i>$ for every $j\in \{1,...,k\}$. Moreover, it follows that for every $j\in \{1,...,k\}$ there exist $v_j\in S$ such that $(x_j,v_j)\in <(i,s)>$. Since $I\cong \Ret(X)$, we have that $\gamma^{j}_{(x_j,s_j)}=\gamma^{j}_{(x_j,v_j)}$ for all $j\in \{1,...,k\}$, hence $(j,t)=\gamma^{1}_{(x_1,v_1)}...\gamma^{k}_{(x_k,v_k)}(i,s)\in <(i,s)>$. By arbitrariness of $s\in S$ and $(j,t)$, we have that $X$ is generated by any element as in the statement.
\end{proof}

\begin{cor}\label{linkret}
    Let $X$ be a finite q-cycle set. If $X$ is irreducible, then so is $\Ret(X)$.\\
    Conversely, if $X$ is indecomposable and $\Ret(X)$ is irreducible, then $X$ is irreducible.
\end{cor}

\begin{proof}
    The first part follows by \cref{riduc}, while the second part follows by \cref{critirr} and \cref{linkret2}.
    %Now, suppose that $X$ is indecomposable and that $\Ret(X)$ is irreducible. By \cref{cov}, without loss of generality, we can suppose that $X=I\times_{\alpha,\alpha'} S$, where $S$ is a set and $I$ is a q-cycle set isomorphic to $\Ret(X)$. Let $(i,s),(j,t)\in X$. Since $X$ is indecomposable, there exist $k\in \mathbb{N}$, $(x_1,s_1),...,(x_k,s_k)\in X$ and $\gamma^{1},...,\gamma^{k}\in \{\sigma,\delta,\sigma^{-1},\delta^{-1} \}$ such that $\gamma^{1}_{(x_1,s_1)}...\gamma^{k}_{(x_k,s_k)}(i,s)=(j,t).$ Now, since $I$ is irreducible, we have that $x_j\in <i>$ for every $j\in \{1,...,k\}$. Moreover, it follows that for every $j\in \{1,...,k\}$ there exist $v_j\in S$ such that $(x_j,v_j)\in <(i,s)>$. Since $I\cong \Ret(X)$, we have that $\gamma^{j}_{(x_j,s_j)}=\gamma^{j}_{(x_j,v_j)}$ for all $j\in \{1,...,k\}$, hence $(j,t)=\gamma^{1}_{(x_1,v_1)}...\gamma^{k}_{(x_k,v_k)}(i,s)\in <(i,s)>$. By arbitrarity of $(i,s)$ and $(j,t)$ and by \cref{critirr}, the irreducibility of $X$ follows.
\end{proof}

\begin{cor}\label{absret}
    Let $X$ be an indecomposable and finite q-cycle set. Then, $X$ is irreducible if and only if the absolute retraction $X$ is irreducible.
\end{cor}

\begin{proof}
    It follows by \cref{linkret}, together with an inductive argument.
\end{proof}

By a standard exercise, one can show that a q-cycle set $X$ is irreducible if and only if the associated solution $(X,r)$ has no sub-solutions. For this reason, we find \cite[Theorem 5.1]{castelli2023studying} as a corollary of \cref{linkret}.

\begin{cor}[Theorem 5.1, \cite{castelli2023studying}]\label{coroll}
    Let $X$ be an indecomposable multipermutation q-cycle set. Then, $X$ is irreducible.
\end{cor}

\begin{proof}
    Since the absolute retraction of a multipermutation q-cycle set has size one, the result follows by \cref{absret}.
\end{proof}

\section{One-generator skew braces and q-cycle sets}\label{secques}

In this section, we study one-generator skew braces. First, we provide general results on the additive/multiplicative group generated by a sub-q-cycle set of the q-cycle set associated to a skew brace. In the main result, we will show that a finite one-generator skew brace $B$ can be characterized by certain sub-q-cycle sets. 
%This answers affirmatively to \cite[Question 6.8]{smock}.

\bigskip

\noindent Below, we provide a concrete description of a finite skew brace generated by one element. Let $B$ be a finite skew brace and $x\in B$. Define $B_0:=\{0,x\}$ and 
$$B_{n}:=\{a\circ b|a,b\in B_{n-1}\}\cup \{\lambda_a(b)|a,b\in B_{n-1}\}$$
for all $n\in \mathbb{N}$. Then, $B(x)$ can be described as in the following result.

\begin{prop}\label{oneg}
    Let $B$ be a finite skew brace and $x\in B$. Then, $B(x)=\bigcup_{n} B_n$.
\end{prop}

\begin{proof}
  First, note that $B_i\subseteq B_j$ for $i\leq j$. Now, if $a,b\in \bigcup_{n} B_n$ then $a^k\in \bigcup_{n} B_n$, for all $k\in \mathbb{N}$, and since $B$ is finite we obtain that $a^{-},-a, a+b\in \bigcup_{n} B_n$. Therefore, $\bigcup_{n} B_n $ is a skew brace containing $x$. Now, by induction on $n$ one can easily show that $\bigcup_{n} B_n $ must be contained in $B(x)$, therefore the statement follows.
\end{proof}

\noindent Before giving the following result, recall that if $B$ is a skew brace, then it has a natural q-cycle set structure given by $a\cdotp b:=\lambda_{a^-}(b)$ and $a:b:=a^{-}\circ b-a^{-}$ for all $a,b,\in B$. From now on, if $S$ is a subset of a skew brace $B$, we indicate by $<S>_{\circ}$ (resp. $<S>_{+}$) the multiplicative (resp. additive) subgroup of $(B,\circ)$ (resp. $(B,+)$) generated by $S$.

\begin{lemma}\label{ugu}
    Let $B$ be a finite skew brace and $X$ a sub-q-cycle set of $B$. Then, $<X>_{+}=<X>_{\circ}$.
\end{lemma}

\begin{proof}
    Since $B$ is finite, every element of $<X>_+$ can be written as $x_1+...+x_n$ for a suitable $n\in \mathbb{N}$ and $x_1,...,x_n\in X$. We show by induction on $n$ that $x_1+...+x_n\in <X>_{\circ}$. If $n=1$, the statement is trivial. Now, $x_1+...+x_n=x_1\circ (\lambda_{x_1^-}(x_2)+...+\lambda_{x_1^-}(x_n))$, and since $X$ is a sub-q-cycle set of $B$ we have $\lambda_{x_1}(x_i)\in X$ for all $i\in \{2,...,n\}$, therefore by inductive hypothesis it follows that there are $t\in \mathbb{N}$, $y_1,...,y_t\in X$ such that  $(\lambda_{x_1^-}(x_2)+...+\lambda_{x_1^-}(x_n))=y_1\circ...\circ y_t$. Hence $x_1+...+x_n=x_1\circ y_1\circ...\circ y_t $ and our claim follows.\\
    Again by finiteness of $B$, every element of $<X>_\circ$ can be written as $x_1\circ...\circ x_n$ for a suitable $n\in \mathbb{N}$ and $x_1,...,x_n\in X$. In the same spirit of the first part, one can easily show by induction on $n$ that $x_1\circ...\circ x_n\in <X>_+$, therefore the statement follows.
\end{proof}

\noindent The following result is implicitly contained in \cite{castelli2023studying}, but here we prefer to give a self-contained proof.

\begin{lemma}\label{transorb}
    Let $B$ be a finite one-generator skew brace and $x\in B$ such that $B=B(x)$. Then, the orbit $X$ of $x$ with respect to the action of the subgroup of $Sym(B)$ generated by the set $\{\sigma_a,\delta_a\mid a\in B\}$, is a transitive cycle base.
\end{lemma}

\begin{proof}
    By \cref{oneg}, it is sufficient to show that $B_n\subseteq <X>_{\circ}$ for all $n\in \mathbb{N}$. We show our claim by induction on $n$. Clearly, $B_0\subseteq <X>_{\circ}$. Now, if $c\in B_{n}$, we have that $c=a\circ b$ or $c=\lambda_a(b)$ with $a,b\in B_{n-1}$. In the first case, by inductive hypothesis, $c\in <X>_{\circ}$. Now, suppose that $c=\lambda_a(b)$. By inductive hypothesis and \cref{ugu}, there exist $n,m\in \mathbb{N}$, $x_1,...,x_m,y_1,...,y_n\in X$ such that $a=x_1\circ...\circ x_m$ and $b=y_1+...+ y_n$. Therefore, we obtain $c=\lambda_{x_1\circ...\circ x_m}(y_1)+...+\lambda_{x_1\circ...\circ x_m}(y_n)$. Since  $\lambda_{x_1\circ...\circ x_n}(y_i)=\sigma_{x_1}^{-1}...\sigma_{x_n}^{-1}(y_i)\in X $ for all $i=1,...,n$, by \cref{ugu} we obtain $c\in <X>_{\circ}$, therefore the statement follows.
\end{proof}

\noindent Now, we provide a lemma of crucial importance for the main theorem of this section. It shows that if $x$ is a generator of a skew brace $B$, then the q-cycle set generated by $x$ generates multiplicatively $B$.

\begin{lemma}\label{genonec}
 Let $B$ be a finite one-generator skew brace, $x\in B$ such that $B=B(x)$, and $C:=<x>$. Then, $B=<C>_{\circ}$.
  %  Let $B$ be a finite one-generator skew brace and $x\in B$ such that $B=B(x)$. Then, every non-zero element of $B$ can be written as $x_1\circ...\circ x_n$ with $x_i\in <x>$ for all $i$.
\end{lemma}
 
\begin{proof}
By \cref{oneg} and the finiteness of $B$, it is sufficient to show, by induction on $n$, that every element of $B_n$ can be written as $x_1\circ...\circ x_k$ for some $k\in \mathbb{N}$ and $x_1,...,x_k\in C$. If $n=0$ the claim is trivial, and if $a,b\in B_{n-1}$, clearly $a\circ b$ can be written in the desired form. It remains to show that $\lambda_a(b)$ can be written as $z_1\circ...\circ z_t$ for some $t\in \mathbb{N}$, and $z_1,...,z_t\in C$. By \cref{ugu},  $b=y_1+...+ y_u$ and $a=x_1\circ...\circ x_v$ for some $u,v\in \mathbb{N}$ and $x_1,...,x_v,y_1,...,y_u\in C$, and hence we have $\lambda_a(b)=\lambda_{x_1}...\lambda_{x_u}(y_1)+...+ \lambda_{x_1}...\lambda_{x_v}(y_u)$. Since $x_i,y_j\in C$ for all $i\in \{1,...,v\},j\in \{1,...,u\}$, by \cref{descr} we have that $\lambda_{x_1}...\lambda_{x_u}(y_i)=\sigma^{-1}_{x_1}...\sigma^{-1}_{x_u}(y_i)\in C $ for all $i\in \{1,...,u\}$, and therefore the claim follows by \cref{ugu}.

\end{proof}

\noindent The following theorem and corollary, that answer affirmatively to \cite[Question 6.8]{smock} and partially extend \cite[Corollary 4.7]{castelli2023studying} in the finite case, provide the precise link between finite one-generator skew left braces and indecomposable q-cycle sets. 

\begin{theor}\label{skewone}
    Let $B$ be a finite skew brace and $X$ an indecomposable sub-q-cycle set of $B$. Then:
    \begin{itemize}
        \item[1)] $B=B(x)$ with $x\in X$ $\Longleftrightarrow$ $X$ is a transitive cycle base with $<x>=X$;
        \item[2)] $B=B(x)$ for all $x\in X $ $\Longleftrightarrow$ $X$ is a transitive cycle base with $X$ irreducible.
        
    \end{itemize}
\end{theor}

\begin{proof}
  First, we prove 1). If $B=B(x)$ with $x\in X$, by \cref{transorb} the orbit of $x$ is a transitive cycle base. Moreover, by \cref{genonec}, every element of $B$ can be written as $x_1\circ...\circ x_k$ for some $k\in \mathbb{N}$ and  $x_1,...,x_k\in <x>$. This implies that $\lambda_{a}(x)\in <x>$ and $\delta_a(x)\in <x>$ for all $a\in B$, therefore $<x>=X$. Conversely, if $X$ is a transitive cycle base with $<x>=X$, then we have that $X\subseteq B(x) $. Indeed, $x\in B(x)$, and if $\mathcal{C}_{n-1}\subseteq  B(x)$ then $-a,a^-\in B(x)$, $\sigma_a(b)=\lambda_{a^-}(b)\in B(x)$ and $\delta_a(b)=\lambda_{a^-}(-a+b+a)\in B(x)$ for all $a,b\in \mathcal{C}_{n-1}$, therefore $\mathcal{C}_n\subseteq B(x)$. Hence, by induction on $n$ we have $X=\bigcup_{n\in \mathbb{N}} \mathcal{C}_{n}\subseteq B(x)$ and this implies $B=<X>_{\circ}\subseteq B(x)$. Part 2) follows by the first one and \cref{critirr}.
\end{proof}

\begin{cor}\label{caraonegen2}
    A skew brace $B$ is a one-generator skew brace if and only if there exists an element $x\in B$ such that the orbit $X$ with respect to the action of the subgroup of $Sym(B)$ generated by the set $\{\sigma_a,\delta_a\mid a\in B\}$, is a transitive cycle base and $X=<x>$. 
\end{cor}

\begin{proof}
    If $B=B(x)$ for some $x\in B$, the necessity follows from \cref{transorb} and \cref{skewone}. Conversely, since transitive cycle bases of $B$ are indecomposable sub-q-cycle sets of $B$, the sufficiency follows by \cref{skewone}.
\end{proof}

\section{Applications on braces and cycle sets}\label{ultsec}

In this section, we focus on braces and cycle sets, and we present examples and results closely related to the ones mainly obtained in \cite{rump2020one} and \cite{smock}.\\
Recall that, given a cycle set $X$, one can define a brace on the permutation group $\mathcal{G}(X)$ generated by the set $\{\sigma_x\mid x\in X\}$, known as the \emph{permutation brace} of $X$. Moreover, if $X$ is irretractable, we have that $\mathcal{G}(X)$ admits a transitive cycle base that is isomorphic to $X$ as a cycle set (see \cite{rump2020} for further details).

\subsection{One generator braces and cycle sets}

As we mentioned in Section $1$, a transitive cycle base is always an indecomposable q-cycle set. In the case of cycle sets, a kind of converse holds.

\begin{lemma}[Proposition 2.12, \cite{smock} - Proposition 9, \cite{rump2020one}]\label{trsmock}
    Every finite indecomposable cycle set $X$ is a transitive cycle base of a finite brace.
\end{lemma}

\noindent The previous lemma, together with our results, yields the following theorem.

\begin{theor}\label{discuss1}
    Let $X$ be a cycle set, and $x\in X$. Then, the following are equivalent.
    \begin{itemize}
        \item[1.] $X=<x>$;
        \item[2.] there is a finite brace $B$ such that $X\subseteq B $,  $X=\{\lambda_b(x)\mid b\in B\}$, and $B=B(x)$.
        %  $X$ is a transitive cycle base of a finite brace $B$, and $B=B(x)$ for all $x\in X$.
    \end{itemize}
\end{theor}

\begin{proof}
    If $X=<x>$, by \cref{descr2} it is indecomposable and by \cref{trsmock} it is the transitive cycle base of a finite brace $B$. By 1) of \cref{skewone} $B=B(x)$.\\
    Conversely, suppose that there is a finite brace $B$ such that $X\subseteq B$, $X=\{\lambda_b(x)\mid b\in B\}$, and $B=B(x)$. By \cref{transorb}, $X$ is a transitive cycle base of $B$ and hence an indecomposable sub-cycle set, therefore by 1) of \cref{skewone} $X=<x>$.   
    %$X$ is a transitive cycle base of a brace $B$, and $B=B(x)$ for all $x\in X$. Then, $X$ is an indecomposable sub-cycle set of $B$ and hence by \cref{skewone} $X$ is irreducible.
\end{proof}

%\begin{cor}\label{discuss1}
%    Let $X$ be a cycle set. Then, the following are equivalent.
%    \begin{itemize}
%        \item[1.] $X=<x>$ for some $x\in X$;
%        \item[2.] $X$ is a transitive cycle base of a finite brace $B$, and $B=B(x)$ for some $x\in X$.
%    \end{itemize}
%\end{cor}

%\begin{proof}
%     If $X=<x>$ for some $x\in X$, then by \cref{trsmock} it is the transitive cycle base of a finite brace $B$, and hence by 1) of \cref{skewone} $B=B(x)$.\\
%    Conversely, suppose that $X$ is a transitive cycle base of a finite brace $B$, and $B=B(x)$ for some $x\in X$. Then, $X$ is an indecomposable sub-cycle set of $B$ and hence by \cref{skewone} $X=<x>$.
%\end{proof}

\noindent In \cite[Theorem 6.5]{smock}, Agata and Alicja Smocktunowitcz showed that a cycle set $X$ is indecomposable and of finite  multipermutation level if and only if there is a brace $B$  having $X$ as a transitive cycle base, such that $B=B(x)$ for all $x\in X$ and $B^{(m)}=\{0\}$ for some $m\in \mathbb{N}$, where $B^{(m)}$ is defined inductively by $B^{(1)}:=B$ and $B^{(m)}:=<\{-a+a\circ b-b\mid a\in B^{(m-1)}, b\in B \}>_{+}$. In \cite[Question 6.7]{smock} they asked if the hypothesis on the multipermutation level can be dropped. By \cite[Proposition 6]{cedo2017yang} and \cite[Lemma 6.4]{smock}, if $X$ does not have finite multipermutation level, then $B^{(m)}\neq \{0\}$ for all $m\in \mathbb{N}$, hence a possible  ``weak'' version of this theorem cannot consider the property $B^{(m)}=\{0\}$. Even with this removal, Theorem 6.5 of \cite{smock} is not true in general: indeed, in \cite{rump2020one} Rump constructed a family $\mathcal{M}$ of cycle sets of size $n^4$, depending on suitable representations of the first Weyl algebra on a field $K$ of size $n$ and of characteristic $2$, such that every brace $B$ having $X\in \mathcal{M}$ as a transitive cycle base satisfies $B\neq B(x)$ for all $x\in X$. Here, we use the same idea of Rump, together with the results of the previous sections, to provide similar examples of size $n^2$. While Rump used a  ``brace-theoretical'' arguments, our proof will be completely different and will be a simple application of the results previously obtained.

\begin{prop}\label{counterex}
    Let $K$ be a field of characteristic $2$ and let $P,Q\in M_2(K)$ given by
$P:=
\begin{bmatrix}
    1       & 1  \\
    0       & 1 
\end{bmatrix}
$ and $Q:=
\begin{bmatrix}
    1       & 0  \\
    1       & 1 
\end{bmatrix}.
$ Let $(K\times K,\cdotp)$ be the binary operation given by $x\cdotp y:=Q^{-1}x+Py$ for all $x,y\in K\times K$. Then, $(K\times K,\cdotp)$ is an indecomposable and irretractable cycle set and
$$<(x_1,x_2)^T>=\begin{cases}\{(x_1,x_2)^T\}\quad if \quad x_1=x_2 \\ \{(x_1,x_2)^T,(x_1,x_1)^T,(x_2,x_2)^T,(x_2,x_1)^T\}\quad if \quad x_1\neq x_2. \end{cases}$$ 
\end{prop}

\begin{proof}
    As in the first part of \cite[Example 2]{rump2020one} one can easily show that $(K\times K,\cdotp)$ is an indecomposable and irretractable cycle set. Using the same notation as in \cref{descr}, we have that 
    $$\mathcal{C}_0:=\{(x_1,x_2)^T \}$$
    $$\mathcal{C}_1:=\{(x_1,x_2)^T,(x_1,x_2)^T\cdotp (x_1,x_2)^T \}=\{(x_1,x_2)^T,(x_2,x_1)^T \}$$
    \begin{eqnarray}
        \mathcal{C}_2&:=&\mathcal{C}_1\cup\{\mathfrak{q}((x_1,x_2)^T),(x_1,x_2)^T\cdotp (x_2,x_1)^T,(x_2,x_1)^T\cdotp (x_1,x_2)^T,\mathfrak{q}((x_2,x_1)^T)\} \nonumber \\
        &=& \{ (x_1,x_2)^T,(x_2,x_1)^T,(x_2,x_2)^T,(x_1,x_1)^T)\}. \nonumber
    \end{eqnarray}
     Now, we have that $\mathcal{C}_n=\mathcal{C}_2$ for all $n>2$. Indeed, by a standard computation one can verify that $\sigma_{(x_1,x_2)^T}(\mathcal{C}_2)=\mathcal{C}_2$ and $\sigma_{(x_1,x_1)^T}(\mathcal{C}_2)=\mathcal{C}_2$; exchanging $x_1$ with $x_2$, we obtain also the equalities $\sigma_{(x_2,x_1)^T}(\mathcal{C}_2)=\mathcal{C}_2$ and $\sigma_{(x_2,x_2)^T}(\mathcal{C}_2)=\mathcal{C}_2$. Therefore, the statement follows by \cref{descr}.
\end{proof}

\noindent The previous proposition allows us to construct a family of indecomposable cycle sets that are not one-generator: indeed, every element generates a sub-cycle set of size at most $4$, hence it is sufficient to consider a cycle set $(K\times K,\cdotp)$ with $|K|>2$. In this way, we obtain the desired examples.

\begin{cor}\label{esemnot}
    Let $X:=(K\times K,\cdotp)$ be the cycle set of \cref{counterex} and let $B$ be a brace such that $X$ is one of its transitive cycle bases. Then, the following are equivalent:
    \begin{itemize}
        \item $|K|>2$;
        \item $B\neq B(x)$ for all $x\in X$.
    \end{itemize}
\end{cor}

\begin{proof}
    It follows by a standard argument, applying  \cref{discuss1} and \cref{counterex}.
\end{proof}

\noindent Even if in general Theorem 6.5 of \cite{smock} is not true when we remove $mpl(X)<+\infty$, we can weaken this hypothesis using the notion of irreducible cycle set. In particular, the following theorem partially extends \cite[Theorem 6.5]{smock}, and answers positively to \cite[Question 6.7]{smock}.

\begin{theor}
    Let $X$ be a cycle set. Then, the following are equivalent.
    \begin{itemize}
        \item[1.] $X$ is irreducible;
        \item[2.] there is a finite brace $B$ such that $X=\{\lambda_b(x)|b\in B\}$ and $B=B(x)$ for all $x\in X$.
    \end{itemize}
\end{theor}

\begin{proof}
    If $X$ is irreducible, then it is indecomposable and by \cref{trsmock} it is the transitive cycle base of a finite brace $B$. By 2) of \cref{skewone}, $B=B(x)$ for all $x\in X$.\\
    Conversely, suppose that there is a finite brace $B$ such that $X=\{\lambda_b(x)|b\in B\}$ and $B=B(x)$ for all $x\in B$. By \cref{transorb}, $X$ is a transitive cycle base of $B$ and hence an indecomposable sub-cycle set, therefore by 2) of \cref{skewone} $X$ is irreducible.   
\end{proof}

Combining \cref{transorb} and \cref{skewone}, we obtain the correspondence between one-generator braces and indecomposable cycle sets.

\begin{cor}\label{transbrace}
      Let $B$ be a finite brace. Then, the following are equivalent.
    \begin{itemize}
        \item[1.] $B$ is a one-generator brace;
        \item[2.] $B$ has a transitive cycle base $X$ such that $X$ is a one-generator cycle set.
    \end{itemize}
\end{cor}

\subsection{Some numerical data}

Using the GAP package \emph{YangBaxter} \cite{YBGAP}, we implemented a GAP code (available upon request to the author) to check when a cycle set $X$ is generated by an element $x\in X$.\\
The algorithm proceeds as follows:
\begin{itemize}
    \item[1)] We give as input a triple $(X,x,[x])$, where $X$ is a cycle set and $x$ is an element of $X$;
    \item[2)] Using the same notation of \cref{descr}, given $\mathcal{C}_{n-1}$ we compute $ \mathcal{C}_{n}$;
    \item[3)] \begin{itemize}
        \item[a)] If $\mathcal{C}_{n-1}\neq \mathcal{C}_{n}$, we go back to step 2) with input $(X,x,\mathcal{C}_n)$.
        \item[b)] If $\mathcal{C}_{n-1}=\mathcal{C}_{n}$, then $<x>=\mathcal{C}_{n}$ and hence we have that $X$ is generated by $x$ if and only if $X=\mathcal{C}_{n}$.        
    \end{itemize}  
\end{itemize}

Among the cycle set of size $<10$, we found only two indecomposable cycle sets that are not one-generator, both of size $8$ (note that  ``almost'' all the indecomposable cycle sets constructed in \cref{counterex} are not one-generator, and the smallest ones have size $16$).  The first one, is given by $X_1:=\{1,2,3,4,5,6,7,8\}$ and 
$$\sigma_1:= (2,3)(4,6)(5,8) \qquad \sigma_5:=(1,2,7,8)(3,4)(5,6) $$
$$\sigma_2:= (1,4)(3,5)(6,7) \qquad \sigma_6:=(1,8,7,2)(3,4)(5,6)$$
$$ \sigma_3:=(1,4)(2,8)(6,7) \qquad  \sigma_7:=(1,2)(3,4,5,6)(7,8)$$
$$ \sigma_4:=(1,7)(2,3)(5,8) \qquad \sigma_8:=(1,2)(3,6,5,4)(7,8).$$
The second one, is given by $X_2:=\{1,2,3,4,5,6,7,8\}$ and 
$$\sigma_1:= (1,2)(3,4,8,6)(5,7) \qquad \sigma_5:=(1,3,7,8)(2,4)(5,6) $$
$$\sigma_2:= (1,2)(3,6,8,4)(5,7) \qquad \sigma_6:=(1,8,7,3)(2,4)(5,6)$$
$$ \sigma_3:= (1,2,7,5)(3,4)(6,8) \qquad  \sigma_7:=(1,3)(2,4,5,6)(7,8)$$
$$ \sigma_4:=(1,5,7,2)(3,4)(6,8) \qquad \sigma_8:=(1,3)(2,6,5,4)(7,8) $$

\noindent Using a combination of the algorithm above descripted and \cref{critirr}, we are able also to check if a cycle set is irreducible. If $X$ is a cycle set with $|X|\in \{2,3,5,6,7\}$, by \cref{coroll}, \cite[Theorem 2.12]{etingof1998set} and \cite[Theorem 4.5]{cedo2022indecomposable}, we have that $X$ is irreducible. In general, every irreducible cycle set is one-generator, but the converse is not necessarily true. If $|X|=4$, we have that all the indecomposable cycle sets are one-generator, but there are two indecomposable cycle sets that are not irreducible: they correspond to the irretractable ones. 
If $|X|=8$, there are $100$ indecomposable cycle sets. Among these, there are $17$ one-generator cycle sets that are not irreducible, and $81$ irreducible cycle sets ($39$ of them are multipermutation cycle sets).
If $|X|=9$, there are $16$ indecomposable cycle sets. We found $15$ irreducible cycle sets ($13$ of them are multipermutation cycle sets) and $1$ one-generator cycle set that is not irreducible.

\medskip

Finally, we turn our attention to braces. By \cref{transbrace}, we have a simple method to determine when a brace $B$ is one-generator. Indeed, one can:
\begin{itemize}
    \item[1)] compute the transitive cycle bases of $B$ among the $\lambda$-orbits (if $B$ has not transitive cycle bases, it is not a one-generator brace).
    \item[2)] check by the algorithm above described if there are transitive cycle bases that are one-generator cycle sets.
\end{itemize}

In this way, we obtained a brace with a transitive cycle base that is not one-generator.

\begin{ex}
    Let $B$ be the permutation brace of $X_2$. Then, $B$ has size $32$ and has $2$ transitive cycle bases, both isomorphic to $X_2$. Hence, by \cref{transbrace}, $B$ is not a one-generator brace.
\end{ex}

Thanks to the cycle sets $X_1$ and $X_2$ and the examples provided in \cref{esemnot}, we can construct braces $B$ having a transitive cycle base $X$ such that $B(x)\neq B$ for all $x\in X$. It is natural asking if this fact allows us to claim that $B$ is not a one-generator brace. The following example shows that the answer is negative.

\begin{ex}
    Let $B$ be the permutation brace associated to $X_1$. Since $X_1$ is not one-generator, $B(x)\neq B$ for all $x\in X_1$. However, $B$ is a one-generator brace. Indeed, inspecting by GAP the transitive cycle bases of $B$, we have that it has a transitive cycle base of size $8$, which is isomorphic to the cycle set $X_3:=\{1,2,3,4,5,6,7,8\}$ given by
    $$\sigma_1:= (1,5,7,8) \qquad \sigma_5:=(1,8,7,5) $$
$$\sigma_2:= (1,4,7,6)(2,8)(3,5) \qquad \sigma_6:=(1,4)(2,5,3,8)(6,7)$$
$$ \sigma_3:=(2,6,3,4) \qquad  \sigma_7:=(1,4)(2,8,3,5)(6,7)$$
$$ \sigma_4:=(2,4,3,6) \qquad \sigma_8:=(1,6,7,4)(2,8)(3,5)$$
and since $X_3$ is a one-generator cycle set (it is generated by $1$), by \cref{caraonegen2} $B$ is a one-generator brace.
\end{ex}

\section{Acknowledgments}

The author is a member of GNSAGA (INdAM), and was partially supported
by the MAD project Cod. $ARS01\_00717$.\\
The author thanks S. Trappeniers for the clarifications on transitive cycle bases related to skew braces with a non-abelian additive group. 

\bibliographystyle{elsart-num-sort}
\bibliography{Bib}

\end{document}